\documentclass[a4paper,twoside,11pt,english,intlimits]{article}

\usepackage{amsmath,amsthm,amsfonts,amssymb}
\usepackage[english]{babel}
\usepackage[T1]{fontenc}
\usepackage{times,euler,euscript,stmaryrd}
\usepackage{fancyhdr,titlesec,url,enumerate}
\usepackage[all]{xy}
\usepackage{hyperref}
\usepackage{ifthen}
\usepackage{tikz}
\usepackage{comment}
\usepackage{wasysym}
\usepackage{multicol}
\usepackage{youngtab}
\CompileMatrices
\usetikzlibrary{calc}

\titleformat{\section}[block]{\scshape\filcenter\Large}{\thesection.}{.5em}{}
\titleformat{\subsection}[block]{\bfseries\filcenter\large}{\thesubsection.}{.5em}{\medskip}
\titleformat{\subsubsection}[runin]{\bfseries}{\thesubsubsection.}{.5em}{}[.]

\titlespacing{\subsubsection}{0pt}{10pt}{.5em}

\newtheoremstyle{ntheorem}%
	{\topsep}{\topsep}{\itshape}{0pt}{\bfseries}{.}{.5em}%
	{\thmnumber{#2.\hspace{.5em}}\thmname{#1}\thmnote{ (#3)}}
	
\newtheoremstyle{ndefinition}%
	{\topsep}{\topsep}{\normalfont}{0pt}{\bfseries}{.}{.5em}%
	{\thmnumber{#2.\hspace{.5em}}\thmname{#1}\thmnote{ (#3)}}
	
\newtheoremstyle{nremark}%
	{\topsep}{\topsep}{\normalfont}{0pt}{\itshape}{.}{.5em}%
	{\thmnumber{}\thmname{#1}\thmnote{ (#3)}}

\theoremstyle{ntheorem}
  	\newtheorem{theorem}[subsubsection]{Theorem}
  	\newtheorem{proposition}[subsubsection]{Proposition}
	\newtheorem{lemma}[subsubsection]{Lemma}

\theoremstyle{ndefinition}

	\newtheorem{example}[subsubsection]{Example}
         \newtheorem{remark}[subsubsection]{Remark}

\fancypagestyle{plain}{
\fancyhf{}\fancyfoot[LC,RC]{}
\fancyfoot[LE,RO]{$\thepage$}

}

\UseTips
\SelectTips{eu}{11}

\newdir{ >}{{}*!/-10pt/@{>}}
\newdir{ -}{{}*!/-10pt/@{}}
\newdir{> }{{}*!/+10pt/@{>}}

\makeatletter

\xyletcsnamecsname@{dir4{}}{dir{}}
\xydefcsname@{dir4{-}}{\line@ \quadruple@\xydashh@}
\xydefcsname@{dir4{.}}{\point@ \quadruple@\xydashh@}
\xydefcsname@{dir4{~}}{\squiggle@ \quadruple@\xybsqlh@}
\xydefcsname@{dir4{>}}{\Tttip@}
\xydefcsname@{dir4{<}}{\reverseDirection@\Tttip@}

\xydef@\quadruple@#1{%
	\edef\Drop@@{%
		\dimen@=#1\relax
		\dimen@=.5\dimen@
		\A@=-\sinDirection\dimen@
		\B@=\cosDirection\dimen@
		\setboxz@h{%
			\setbox2=\hbox{\kern3\A@\raise3\B@\copy\z@}%
			\dp2=\z@ \ht2=\z@ \wd2=\z@ \box2
			\setbox2=\hbox{\kern\A@\raise\B@\copy\z@}%
			\dp2=\z@ \ht2=\z@ \wd2=\z@ \box2
			\setbox2=\hbox{\kern-\A@\raise-\B@\copy\z@}%
			\dp2=\z@ \ht2=\z@ \wd2=\z@ \box2
			\setbox2=\hbox{\kern-3\A@\raise-3\B@ \noexpand\boxz@}%
			\dp2=\z@ \ht2=\z@ \wd2=\z@ \box2
		}%
		\ht\z@=\z@ \dp\z@=\z@ \wd\z@=\z@ \noexpand\styledboxz@
	}%
}

\xydef@\Tttip@{\kern2pt \vrule height2pt depth2pt width\z@
	\Tttip@@ \kern2pt \egroup
	\U@c=0pt \D@c=0pt \L@c=0pt \R@c=0pt \Edge@c={\circleEdge}%
	\def\Leftness@{.5}\def\Upness@{.5}%
	\def\Drop@@{\styledboxz@}\def\Connect@@{\straight@{\dottedSpread@\jot}}}
	
\xydef@\Tttip@@{%
	\dimen@=.25\dimen@
 	\B@=\cosDirection\dimen@
	\setboxz@h\bgroup\reverseDirection@\line@ \wdz@=\z@ \ht\z@=\z@ \dp\z@=\z@
	{\vDirection@(1,-1)\xydashl@ \xyatipfont\char\DirectionChar}%
	{\vDirection@(1,+1)\xydashl@ \xybtipfont\char\DirectionChar}%
}

\xydef@\ar@form{
	\ifx \space@\next \expandafter\DN@\space{\xyFN@\ar@form}%
	\else\ifx ^\next \DN@ ^{\xyFN@\ar@style}\edef\arvariant@@{\string^}%
	\else\ifx _\next \DN@ _{\xyFN@\ar@style}\edef\arvariant@@{\string_}%
	\else\ifx 0\next \DN@ 0{\xyFN@\ar@style}\def\arvariant@@{0}%
	\else\ifx 1\next \DN@ 1{\xyFN@\ar@style}\def\arvariant@@{1}%
	\else\ifx 2\next \DN@ 2{\xyFN@\ar@style}\def\arvariant@@{2}%
	\else\ifx 3\next \DN@ 3{\xyFN@\ar@style}\def\arvariant@@{3}%
	\else\ifx 4\next \DN@ 4{\xyFN@\ar@style}\def\arvariant@@{4}%
	\else\ifx \bgroup\next \let\next@=\ar@style
	\else\ifx [\next \DN@[##1]{\ar@modifiers{[##1]}}
	\else\ifx *\next \DN@ *{\ar@modifiers}%
	\else\addLT@\ifx\next \let\next@=\ar@slide
	\else\ifx /\next \let\next@=\ar@curveslash
	\else\ifx (\next \let\next@=\ar@curveinout 
	\else\addRQ@\ifx\next \addRQ@\DN@{\ar@curve@}%
	\else\addLQ@\ifx\next \addLQ@\DN@{\xyFN@\ar@curve}%
	\else\addDASH@\ifx\next \addDASH@\DN@{\defarstem@-\xyFN@\ar@}%
	\else\addEQ@\ifx\next \addEQ@\DN@{\def\arvariant@@{2}\defarstem@-\xyFN@\ar@}%
	\else\addDOT@\ifx\next \addDOT@\DN@{\defarstem@.\xyFN@\ar@}%
	\else\ifx :\next \DN@:{\def\arvariant@@{2}\defarstem@.\xyFN@\ar@}%
	\else\ifx ~\next \DN@~{\defarstem@~\xyFN@\ar@}%
	\else\ifx !\next \DN@!{\dasharstem@\xyFN@\ar@}%
	\else\ifx ?\next \DN@?{\ar@upsidedown\xyFN@\ar@}%
	\else \let\next@=\ar@error
	\fi\fi\fi\fi\fi\fi\fi\fi\fi\fi\fi\fi\fi\fi\fi\fi\fi\fi\fi\fi\fi\fi\fi \next@}

\makeatother

\newcommand{\fl}{\to}
\newcommand{\dfl}{\Rightarrow}

\newcommand{\qfl}{\xymatrix@1@C=10pt{\ar@4 [r] &}}

\newcommand{\opfl}[1]{\xymatrix @C=1.5em {\strut \ar@{->>} [r] ^-{#1} & \strut}}
\newcommand{\odfl}[1]{\xymatrix @C=1.5em {\strut \ar@2 [r] ^-{#1} & \strut}}
\newcommand{\otfl}[1]{\xymatrix @C=1.5em {\strut \ar@3 [r] ^-{#1} & \strut}}
\newcommand{\oqfl}[1]{\xymatrix @C=1.5em {\strut \ar@4 [r] ^*+{#1} & \strut}}

\newcommand{\ie}{\emph{i.e.}}
\renewcommand{\phi}{\varphi}
\renewcommand{\epsilon}{\varepsilon}

\definecolor{orange}{rgb}{1,0.55,0}
\definecolor{vert}{rgb}{0,0.45,0}

\newcommand{\ifthen}[2]{\ifthenelse{#1}{#2}{}}

\renewcommand{\fl}{\rightarrow}
\renewcommand{\to}{\longrightarrow}

\DeclareMathOperator{\Path}{Path}

\def\P{\mathbf{P}}
\begin{document}

\thispagestyle{empty}

\begin{center}

\begin{Large}\begin{uppercase}
{Finite convergent presentations of plactic monoids for semisimple Lie algebras \footnote{This work is partially supported by the French National Research Agency, ANR-13-BS02-0005-02.}}
\end{uppercase}\end{Large}

\bigskip\hrule height 1.5pt \bigskip

\begin{large}\begin{uppercase}
{Nohra Hage}
\end{uppercase}\end{large}
\vspace{2cm}

\begin{small}\begin{minipage}{14cm}
\noindent\textbf{Abstract --} We study rewriting properties of  the column presentation of plactic monoid for any semisimple Lie algebra such as termination and confluence. Littelmann described this presentation using L-S paths generators. Thanks to the shapes of  tableaux,  we show that this presentation is finite and convergent. We obtain as a corollary that plactic monoids for any semisimple Lie algebra satisfy homological finiteness properties.

\smallskip\noindent\textbf{M.S.C. 2010 -- } 16S15, 68Q42, 20M05, 68R05, 06B15
\end{minipage}\end{small}
\end{center}

%
%
%
%
%

\section{Introduction}
Using his path model, Littelmann defined in~\cite{Littelmann96} a plactic algebra for any semisimple Lie algebra. As a consequence, he gave some presentations by generators and relations of the plactic algebra of types~A,~B,~C,~D and~$G_2$.  Using a case-by-case analysis, the plactic congruence can be also  defined using Kashiwara's theory of crystal bases~\cite{JimboMisraMiwaOkado90, Kashiwara91, KashiwaraNakashima94, Kashiwara94}.  Lascoux, Leclerc and Thibon presented in~\cite{LascouxLeclercThibon95} the plactic monoid of type~A using the theory of crystal bases and  gave a presentation of the plactic monoid of type~C without proof. Shortly after that, Lecouvey in~\cite{Lecouvey02} and Baker in~\cite{Baker00} described independently the plactic monoid of type~C using also Kashiwara's theory of crystal bases. Plactic monoids of types~B,~D and~$G_2$ were introduced by Lecouvey, see~\cite{Lecouvey07}. 

The plactic monoid of rank~$n$ introduced by Lascoux and Sch\"{u}tzenberger in~\cite{LascouxSchutzenberger81}, corresponds to the representations of the general linear Lie algebra of $n$ by $n$ matrices.  This Lie algebra is of type~A, so the corresponding plactic monoid is known as the plactic monoid of type~A, and denoted by~$\P_n(A)$. 

The monoid~$\P_n(A)$  has found several applications in algebraic combinatorics and representation theory~\cite{LascouxSchutzenberger81,LascouxLeclercThibon95,Fulton97,Lothaire02}. Similarly, plactic monoids for any types have many applications like the combinatorial description of the Kostka-Foulkes polynomials which arise as entries of the character table of the finite reductive groups~\cite{Lecouvey06}.
\medskip

More recently, plactic monoids were investigated by rewriting methods. In~\cite{KubatOkninski14},  Kubat and Okninski showed that for $n>3$, there is no finite completion of the presentation of the  monoid~$\P_n(A)$  with the Knuth generators. Bokut et al. in~\cite{BokutChenChenLi15}  and Cain et al. in~\cite{CainGrayMalheiro15} constructed independently finite convergent presentation of the monoid~$\P_n(A)$ by adding columns generators. This presentation is called the \emph{column presentation}. Having a finite convergent presentation of the monoid~$\P_n(A)$ had many consequences. In particular, the column presentation was used by Lopatkin in~\cite{Lopatkin14} to construct Anick's resolution for the monoid~$\P_n(A)$ and by the author and Malbos in~\cite{HageMalbos14} to construct coherent presentations of it. Note that coherent presentation extends the notion of a presentation of the monoid by homotopy generators taking into account the relations among the relations. 
Using Kashiwara's theory of crystal bases, the author generalized the column presentation for type~A and  constructed a finite convergent presentation  of  plactic monoid of type~C by adding admissible column generators~\cite{Hage14}. A bit later, a similar column presentation was obtained by Cain et al. for plactic monoids of types~B,~C,~D and~$G_2$~\cite{CainGrayMalheiro14}.
\medskip

In the present work, we consider plactic monoids  for any semisimple Lie algebra. The column presentation of these monoids was introduced by  Littelmann in~\cite{Littelmann96} using  L-S paths. An L-S path corresponds  to a column for type~A, and to an admissible column for types~C,~B,~D and $G_2$ in the Lecouvey sense, see~\cite{Lecouvey07}.   We study the column presentation using rewriting methods. For this, we consider a rewriting system where the generating set contains  the finite set of L-S paths. The right-hand sides of rewriting rules are standard tableaux. The rewriting system rewrites two L-S paths that their concatenation do not form a standard tableau to their corresponding standard tableau form. Using the shapes of tableaux, we show that this presentation is finite and convergent.  As a consequence,  we deduce that plactic monoids for any types satisfy  some homological finiteness properties.

Note that the convergent column presentation of  plactic monoids for any semisimple Lie algebra given in this paper using L-S paths coincides with the presentations constructed for type~A in~\cite{CainGrayMalheiro15,BokutChenChenLi15}, for type~C in~\cite{Hage14} and for types~B,~C,~D and $G_2$ in~\cite{CainGrayMalheiro14}.
\medskip

The paper is organised as follows. In Section~\ref{placticalgebra},  we recall the definitions of paths, root operators and L-S paths.  After we present the definition and  some properties of standard tableaux as defined by Littelmann and we recall the  definition of the plactic algebra for any semisimple Lie algebra. In Section~\ref{finiteconvergentpresentation}, we recall some rewriting properties of the presentations of monoids in term of polygraphs. After,  we show that the column  presentation of the plactic monoid for any semisimple Lie algebra  is finite and convergent.

\section{Plactic algebra}
\label{placticalgebra}

\subsection{L-S Paths}
In this section, we recall the definitions and properties of paths, root operators and L-S paths.  We refer the reader to~\cite{Littelmann94,Littelmann95,Littelmann96} for a full introduction.

\subsubsection{Root system}
Let $\mathfrak{g}$ be a semisimple Lie algebra. Let $V$ be the finite vector space with standard inner product $\langle \cdot, \cdot \rangle $ spanned by $\Phi\subset V\setminus \{0\}$ the root system of~$\mathfrak{g}$.  Let $\Phi^{+}$ be the set of its positive roots. A positive root $\alpha$ in $\Phi^+$ is \emph{simple}, if it can not be written as $\alpha_{1}+\alpha_{2}$, where $\alpha_{1},\alpha_{2}\in\Phi^+$. 
For a root $\alpha$, define by
$\alpha^\vee := {2\alpha\over \langle\alpha,\alpha\rangle}\,$  its \emph{coroot}. Denote by  
\[
X
\: = \:
\big\{\; v\in V \; \big| \; \langle  v, \alpha^{\vee}_i\rangle \in \mathbb{Z}, \text{ for all } i \;\big\}
\]
the weight lattice of the Lie algebra $\mathfrak{g}$ and set $X_{\mathbb{R}} := X\otimes_{\mathbb{Z}} \mathbb{R}$. A \emph{fundamental weight} $\omega_i$ in $X$  satisfy
$\langle \omega_i, \alpha_{j}^{\vee}\rangle = \delta_{ij}$,
for all $i$ and $j$. 
The set of dominant weights is   
\[X^{+} \:=\: \big\{\; \lambda\in X \;\big|\; \langle \lambda,\alpha^{\vee}_{i}\rangle \geq 0, \text{ for all } i  \;\big\}\]
and the dominant chamber is
\[X^{+}_{\mathbb{R}} \:=\: \big\{\; \lambda\in X_{\mathbb{R}} \;\big|\; \langle \lambda,\alpha^{\vee}_{i}\rangle \geq 0, \text{ for all } i \;\big\}.\]

\begin{example}
Let $\mathfrak{g} = \mathfrak{sl}_{3}$. Consider $V= \big\{\;(x_1,x_2,x_3)\in\mathbb{R}^{3}\;\big|\; x_1+x_2+x_3 = 0\;\big\}$ and let $\{\epsilon_{1},\epsilon_{2},\epsilon_{3}\}$ be the canonical basis of $\mathbb{R}^{3}$. The simple roots of $\mathfrak{g}$ are $\alpha_{1} = \epsilon_{1}-\epsilon_{2}$ and $\alpha_{2}= \epsilon_{2}-\epsilon_{3}$. Its fundamental weights are $\omega_{1} = \epsilon_{1}$ and $\omega_{2}= \epsilon_{1}+ \epsilon_{2}$ (we still denote by $\epsilon_i$ the projection of $\epsilon_i$ into $V$). An example of dominant weight is $\omega_1 + \omega_2$. The dominant chamber is the hatched area on the following figure:
\begin{center}
\label{rootfigure}
\begin{tikzpicture}[scale=1]
\draw[-, dotted] (-2,1.73205080756) --(2,1.73205080756);
\draw[-, dotted] (-2,0) --(2,0);
\draw[-, dotted] (-2,0.86) --(2,0.86);
\draw[-, dotted] (-2,-0.86) --(2,-0.86);
\draw[->, dotted] (0,0) --(0.5, 0.86602540378);
\node[above] at (0.5, 0.86602540378){\small{$\varepsilon_{1} =\omega_{1}$}};
\draw[->,dotted] (0,0) --(-1, 0);
\draw[-,dotted] (1,-1.73205080756) --(-1, 1.73205080756);
\draw[-,dotted] (2,-1.73205080756) --(0, 1.73205080756);
\draw[-,dotted] (0,-1.73205080756) --(-2, 1.73205080756);
\node[left] at (-1, 0){\small{$\varepsilon_{2}$}};
\draw[->, dotted] (0,0) --(0.5, -0.86602540378);
\draw[-, dotted] (1,1.73205080756) --(-1, -1.73205080756);
\draw[-, dotted] (2,1.73205080756) --(0, -1.73205080756);
\draw[-, dotted] (0,1.73205080756) --(-2, -1.73205080756);
\node[below] at (0.5, -0.86602540378) {\small{$\varepsilon_{3}$}};
\draw[->, blue] (0,0) -- (1.5,0.86602540378);
\node[right] at (1.5,0.86602540378) {\small{$\alpha_{1}$}};
\node[above] at (0, 1.73205080756) {\small{$\omega_{1}+\omega_{2}$}}; 
\node[left] at (-0.5, 0.86602540378) {\small{$\omega_{2}$}};
\draw[->, red] (0,0) -- (-1.5,0.86602540378);
\node[left] at (-1.5,0.86602540378) {\small{$\alpha_{2}$}};
\node[below] at (0,0) {$0$};
\fill[red!20!, opacity = 0.5] (0,0) -- (1,1.73205080756) -- (-1,1.73205080756) -- cycle;
\end{tikzpicture}
\end{center}
\end{example}

\begin{example}
Let $\mathfrak{g} = \mathfrak{sp}_{4}$. Consider $V=\mathbb{R}^{2}$ with its canonical basis $\{\epsilon_{1},\epsilon_{2}\}$. The simple roots are $\alpha_{1} = \epsilon_{1}-\epsilon_{2}$ and $\alpha_{2}= 2\epsilon_{2}$. The fundamental weights are $\omega_{1} = \epsilon_{1}$ and $\omega_{2}= \epsilon_{1}+ \epsilon_{2}$. An example of  dominant weight is $\omega_1 + \omega_2$. The dominant chamber is the hatched area on the following figure:
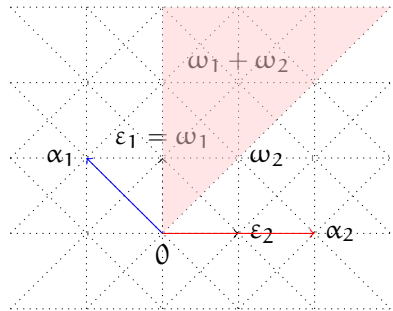
\begin{center}
\label{rootfigure}
\begin{tikzpicture}[scale=1]
\draw[-, dotted] (-2,3) --(3,3);
\draw[-, dotted] (-2,0) --(3,0);
\draw[-, dotted] (-2,2) --(3,2);
\draw[-, dotted] (-2,1) --(3,1);
\draw[-, dotted] (-2,-1) --(3,-1);
\draw[-, dotted] (0,-1) --(0,3);
\draw[-, dotted] (1,-1) --(1,3);
\draw[-, dotted] (2,-1) --(2,3);
\draw[-, dotted] (-1,-1) --(-1,3);
\draw[-, dotted] (-1,-1) --(3,3);
\draw[-, dotted] (0,-1) --(3,2);
\draw[-, dotted] (1,-1) --(3,1);
\draw[-, dotted] (-2,-1) --(2,3);
\draw[-, dotted] (-2,0) --(1,3);
\draw[-, dotted] (-2,1) --(0,3);
\draw[-, dotted] (-2,3) --(2,-1);
\draw[-, dotted] (-2,1) --(0,-1);
\draw[-, dotted] (-2,2) --(1,-1);
\draw[-, dotted] (-1,3) --(3,-1);
\draw[-, dotted] (0,3) --(3,0);
\draw[-, dotted] (1,3) --(3,1);
\node[below] at (0,0) {$0$};
\node[right] at (1, 0){$\varepsilon_{2}$};
\draw[->, dotted] (0,0) --(1,0);
\draw[->, red] (0,0) -- (2,0);
\node[right] at (2,0) {\small{$\alpha_{2}$}};
\node[above] at (0, 1){\small{$\varepsilon_{1} =\omega_{1}$}};
\draw[->, dotted] (0,0) --(0,1);
\node[left] at (-1,1) {\small{$\alpha_{1}$}};
\draw[->, blue] (0,0) -- (-1,1);
\node[right] at (1, 1) {\small{$\omega_{2}$}};
\node[above] at (1, 2) {\small{$\omega_{1}+\omega_{2}$}};
\fill[red!20!, opacity = 0.5] (0,0) -- (0,3) --(3,3)-- cycle;
\end{tikzpicture}
\end{center}
\end{example}
For more informations, we refer the reader to~\cite{BourbakiLie4-6,Humphreys78}.

\subsubsection{Paths}
A \emph{path} is a piecewise linear, continuous  map $\pi:[0,1]\to X_{\mathbb{R}}$.  We will consider  paths up to a reparametrization, that is, the path $\pi$ is equal to any path $\pi\circ \varphi$, where $\varphi:[0,1]\to [0,1]$ is a piecewise linear, nondecreasing, surjective, continuous  map. The \emph{weight} of a path $\pi$ is $\textrm{wt}(\pi)=\pi(1)$.
For example, for the Lie algebra $\mathfrak{g}=\mathfrak{sl}_{n}$, the paths $\pi_{\epsilon_i}: t\mapsto t\epsilon_{i}$ are of weight $\epsilon_{i}$, for $i=1,\ldots,n$.
Denote by 
\[
\Pi= \big\{\; \pi: [0,1]\to X_{\mathbb{R}} \; \big| \; \pi(0) = 0 \text{ and } \pi(1)\in X \; \big\}
\]
the set of all paths with sources $0$ such that their weight lies in  $X$. 
Given two paths $\pi_1$ and $\pi_2$ in $\Pi$, the \emph{concatenation} of $\pi_1$ and $\pi_2$, denoted by $\pi:=\pi_1\star\pi_2$,  is defined by:
\[
\pi(t) :=\left \{
\begin{array}{ll}
\pi_{1}(2t)  &  \text{ for } 0\leq t\leq \frac{1}{2} \\
\pi_{1}(1)+\pi_{2}(2t-1)&   \text{ for }\frac{1}{2}\leq t\leq 1 \\
\end{array}
\right. 
\]
Denote by $\mathbb{Z}\Pi$ the \emph{algebra of paths} defined as the free $\mathbb{Z}$-module with basis $\Pi$ whose product is given by the concatenation of paths and where the unity is the trivial path
\[\begin{array}{rll}
\theta :& [0,1]&\to X_{\mathbb{R}}\\
&t&\mapsto 0.
\end{array}\]

\begin{example}
Let $\mathfrak{g}=\mathfrak{sl}_{3}$. Consider the paths $\pi_1: t\mapsto t\epsilon_{1}$ and $\pi_{2}:t\mapsto t\epsilon_{2}$. The path $\pi_{1}\star \pi_{2}$ is the path on the following figure:
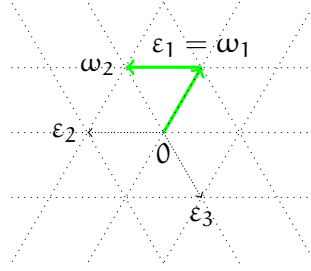
\begin{center}
\label{rootfigure}
\begin{tikzpicture}[scale=1]
\draw[-, dotted] (-2,0) --(2,0);
\draw[-, dotted] (-2,0.86) --(2,0.86);
\draw[-, dotted] (-2,-0.86) --(2,-0.86);
\draw[->, green,very thick] (0,0) --(0.5, 0.86602540378);
\draw[->, green,very thick] (0.5, 0.86602540378) -- (-0.5, 0.86602540378);
\node[above] at (0.5, 0.86602540378){$\varepsilon_{1} =\omega_{1}$};
\draw[->,dotted] (0,0) --(-1, 0);
\draw[-,dotted] (1,-1.73205080756) --(-1, 1.73205080756);
\draw[-,dotted] (2,-1.73205080756) --(0, 1.73205080756);
\draw[-,dotted] (0,-1.73205080756) --(-2, 1.73205080756);
\node[left] at (-1, 0){$\varepsilon_{2}$};
\draw[->, dotted] (0,0) --(0.5, -0.86602540378);
\draw[-, dotted] (1,1.73205080756) --(-1, -1.73205080756);
\draw[-, dotted] (2,1.73205080756) --(0, -1.73205080756);
\draw[-, dotted] (0,1.73205080756) --(-2, -1.73205080756);
\node[below] at (0.5, -0.86602540378) {$\varepsilon_{3}$};
\node[left] at (-0.5, 0.86602540378) {\small{$\omega_{2}$}};
\node[below] at (0,0) {$0$};
\end{tikzpicture}
\end{center}
\end{example}

\subsubsection{Root operators}
For each simple root $\alpha$, one defines  root operators $e_{\alpha}, f_{\alpha}: \Pi \fl \Pi\cup\{0\}$ as follows. Every path $\pi$ in $\Pi$ is cutted into three parts, \ie, $\pi= \pi_{1}\star \pi_{2}\star \pi_{3}$. Then the new path $e_{\alpha}(\pi)$ or $f_{\alpha}(\pi)$ is either equal to $0$ or $\pi_{1}\star s_{\alpha}(\pi_{2})\star \pi_{3}$, where $s_{\alpha}$ denote the simple reflection with respect to the root $\alpha$. In other words, consider the function
\[
\begin{array}{rrl}
h_{\alpha}&:[0,1]&\to \mathbb{R}\\
 &t&\mapsto \langle \pi(t),\alpha^{\vee} \rangle
\end{array}\]
Let $Q:=\text{min}(Im( h_{\alpha})\cap \mathbb{Z})$ be the minimum attained by $h_{\alpha}$.  If $Q=0$, define $e_{\alpha}(\pi)= 0 $. If $Q<0$, let 
\[t_{1} = min\{\;t\in[0,1]\;\big|\; h_{\alpha}(t) = Q\;\}\]
and
\[t_{0} = \text{max}\{\;t<t_{1}\;\big|\; h_{\alpha}(t) = Q+1\;\}.\]
Denote by $\pi_{1}, \pi_{2}$ and $\pi_{3}$ the paths defined by 
\[
\begin{array}{rl}
\pi_{1}(t):=& \pi(tt_{0})\\
\pi_{2}(t):=&\pi(t_{0}+t(t_{1}-t_{0})) - \pi(t_{0})\\
\pi_{3}(t):=&\pi(t_{1}+t(1-t_{1}))-\pi(t_{1}), \text{ }\text{for}\text{ } t\in [0,1].
\end{array}
\]
By definition of the $\pi_{i}$, we have $\pi= \pi_{1}\star \pi_{2}\star \pi_{3}$. Then $e_{\alpha}(\pi) =  \pi_{1}\star s_{\alpha}(\pi_{2})\star \pi_{3}$.

Similarly, one can define the operator $f_{\alpha}$. Let
\[
p= \text{max}\{\;t\in [0,1]\;\big|\; h_{\alpha}(t) = Q\;\}.
\]
Denote by  $P$ the integral part of $h_{\alpha}(1)-Q$. If $P=0$, define $f_{\alpha}(\pi)=0$. If $P>0$, let $x>p$ such that 
\[
h_{\alpha}(x) = Q+1  \text{ } \text{ and } \text{ } Q<h_{\alpha}(x)< Q+1, \text{} \text{ for} \text{ } p<t<x.
\]
Denote by $\pi_{1}, \pi_{2}$ and $\pi_{3}$ the paths defined by 
\[
\begin{array}{rl}
\pi_{1}(t):=& \pi(tp)\\
\pi_{2}(t):=&\pi(p+t(x-p)) - \pi(p)\\
\pi_{3}(t):=&\pi(x+t(1-x))-\pi(x), \text{ }\text{for}\text{ } t\in [0,1].
\end{array}
\]
By definition of the $\pi_{i}$, we have $\pi= \pi_{1}\star \pi_{2}\star \pi_{3}$. Then $f_{\alpha}(\pi) =  \pi_{1}\star s_{\alpha}(\pi_{2})\star \pi_{3}$.
\medskip

These operators preserve the length of the paths. We have also that if $f_{\alpha}(\pi)=\pi^{'}\neq 0$ then $e_{\alpha}(\pi^{'})~=~\pi\neq 0$ and $\textrm{wt}(f_{\alpha}(\pi)) = \textrm{wt}(e_{\alpha}(\pi)) - \alpha$.
\medskip

For all simple root $\alpha$, let $\mathcal{A}$ be the subalgebra of $\text{End}_{\mathbb{Z}}(\mathbb{Z}\Pi)$ generated by the root operators $f_{\alpha}$ and $e_{\alpha}$. Define $\Pi^{+}$ to be the set of paths $\pi$ such that the image is contained in $X_{\mathbb{R}}^{+}$ and denote by $M_{\pi}$ the $\mathcal{A}$-module $\mathcal{A}\pi$. Let $B_{\pi}$  be the $\mathbb{Z}$-basis  $M_{\pi}\cap\Pi$ of $M_{\pi}$. In other words,  for a finite set~$I$,  we have
\[B_{\pi} \:=\: \big\{\;f_{i_{1}}\ldots f_{i_{r}}(\pi)\;\big|\; \pi\in\Pi^{+}\text{ and } i_{1},\ldots, i_{r}\in I \;\big\}.\]

For a dominant weight $\lambda$,  consider  the path $\pi_{\lambda}(t) = t\lambda$
that connects the origin with $\lambda$ by a straight line.  Denote by $M_{\lambda}$ the $\mathcal{A}$-module $\mathcal{A}\pi_{\lambda}$ generated by the path $\pi_{\lambda}$. In addition, the $\mathbb{Z}$-module $M_{\lambda}$ has for a basis the set $B_{\pi_{\lambda}}$ consisting of all paths in $M_{\lambda}$.

Given two paths $\pi$ and  $\pi^{'}$ in $\Pi^{+}$, if $\pi(1)=\pi^{'}(1)$, then the $\mathcal{A}$-modules $M_{\pi}$ and $M_{\pi^{'}}$ are isomorphic. 
For $\pi$ in $\Pi^{+}$, let $\eta$ in $M_{\pi}$ be an arbitrary path. The minimum of the function $h_{\alpha}(t) = \langle \pi(t),\alpha^{\vee} \rangle$ is an integer for all simple roots. In addition,   we have $e_{\alpha}(\eta) = 0$  if and only if $\eta = \pi$.

\begin{example}
 \label{rootoperatorexample}
Let $\mathfrak{g}=\mathfrak{sl}_{3}$ and let $\alpha_{1},\alpha_{2}$ be its simple roots and $\omega_1,\omega_2$ be its fundamental weights. For $\lambda =\omega_{1}+\omega_{2}$, consider the path $\pi_{\lambda}:t\mapsto t\lambda$. Let us compute $B_{\pi_{\lambda}}$. We have 
\begin{center}
\begin{tikzpicture}[scale=1.5]
\draw[-, dotted] (-2,0) --(2,0);
\draw[-, dotted] (-2,0.86) --(2,0.86);
\draw[-, dotted] (-2,-0.86) --(2,-0.86);
\draw[->, dotted,] (0,0) --(0.5, 0.86602540378);
\node[right] at (0.5, 0.86602540378){\tiny{$\omega_{1}$}};
\draw[->,dotted] (0,0) --(-1, 0);
\draw[-,dotted] (1,-1.73205080756) --(-1, 1.73205080756);
\draw[-,dotted] (2,-1.73205080756) --(0, 1.73205080756);
\draw[-,dotted] (0,-1.73205080756) --(-2, 1.73205080756);
\node[left] at (-1, 0){\tiny{$\varepsilon_{2}$}};
\draw[->, dotted] (0,0) --(0.5, -0.86602540378);
\draw[-, dotted] (1,1.73205080756) --(-1, -1.73205080756);
\draw[-, dotted] (2,1.73205080756) --(0, -1.73205080756);
\draw[-, dotted] (0,1.73205080756) --(-2, -1.73205080756);
\node[below] at (0.5, -0.86602540378) {\tiny{$\varepsilon_{3}$}};
\draw[->, blue,very thick] (0,0) -- (1.5,0.86602540378);
\node[above] at (1.5,0.86602540378) {\tiny{\textcolor{blue}{$f_{\alpha_{2}}(\pi_{\lambda}$)}}};
\node[right] at (1.5,0.86602540378) {\tiny{$\alpha_{1}$}};
\draw[-,green,very thick] (0,0)--(0.75,-0.4330127189);
\draw[-,green,very thick](0.75,-0.4330127189)-- (0.76, -0.4);
\draw[->,green,very thick](0.76, -0.4)--(0,0);
\node[right] at (0.4,-0.2330127189) {\tiny{\textcolor{green}{$f_{\alpha_{2}}(f_{\alpha_{1}}(\pi_{\lambda}$))}}};
\draw[->,gray](0,0)--(1.5, -0.86602540378);
\node[right] at (1.5, -0.86602540378) {\tiny{\textcolor{gray}{$f_{\alpha_{2}}(f_{\alpha_{2}}(f_{\alpha_{1}}(\pi_{\lambda}$)))}}};
\draw[-,cyan,very thick] (0,0)--(-0.75,-0.4330127189);
\draw[-,cyan,very thick](-0.75,-0.4330127189)-- (-0.76, -0.4);
\draw[->,cyan,very thick](-0.76, -0.4)--(0,0);
\node[left] at (-0.4,-0.2330127189)  {\tiny{\textcolor{cyan}{$f_{\alpha_{1}}(f_{\alpha_{2}}(\pi_{\lambda}$))}}};
\draw[->,magenta](0,0)--(-1.5, -0.86602540378);
\node[left] at (-1.5, -0.86602540378) {\tiny{\textcolor{magenta}{$f_{\alpha_{1}}(f_{\alpha_{1}}(f_{\alpha_{2}}(\pi_{\lambda}$)))}}};
\draw[->, black,very thick] (0,0) -- (0, 1.73205080756);
\node[above] at (0, 1.73205080756) {\tiny{$\lambda$}}; 
\node[right] at (0, 1) {\tiny{$\pi_{\lambda}$}};
\node[left] at (-0.5, 0.86602540378) {\tiny{$\omega_{2}$}};
\draw[->, red,very thick] (0,0) -- (-1.5,0.86602540378);
\node[above] at (-1.5,0.86602540378 ) {\tiny{\textcolor{red}{$f_{\alpha_{1}}$($\pi_{\lambda}$)}}};
\node[left] at (-1.5,0.86602540378) {\tiny{$\alpha_{2}$}};
\draw[->, black,very thick] (0,0) -- (0, -1.73205080756);
\node[above] at (0, -1.73205080756) {\tiny{$f_{\alpha_{2}}(f_{\alpha_{1}}(f_{\alpha_{1}}(f_{\alpha_{2}}(\pi_{\lambda}$))))}};
\fill[red!20!, opacity = 0.5] (0,0) -- (1,1.73205080756) -- (-1,1.73205080756) -- cycle;
\end{tikzpicture}
\end{center}
where
\[\begin{array}{rl}
\pi_{\lambda}:&t\to t\lambda\\
f_{\alpha_1}(\pi_{\lambda})=:\pi_{2}:&t\to t\alpha_{2}\\
f_{\alpha_2}(\pi_{\lambda})=:\pi_{3}:&t\to t\alpha_{1}\\
f_{\alpha_2}(\pi_2)=:\pi_{4}:&t\to \left \{
\begin{array}{ll}
-t\alpha_{2}  &  \text{ for } 0\leq t\leq \frac{1}{2} \\
(t-1)\alpha_{2}&   \text{ for }\frac{1}{2}\leq t\leq 1 \\
\end{array}
 \right. \\
f_{\alpha_1}(\pi_3)=:\pi_{5}:&t\to \left \{
\begin{array}{ll}
-t\alpha_{1}  &  \text{ for } 0\leq t\leq \frac{1}{2} \\
(t-1)\alpha_{1}&   \text{ for }\frac{1}{2}\leq t\leq 1 \\
\end{array}
 \right. \\
f_{\alpha_2}(\pi_4)=:\pi_{6}:&t\to -t\alpha_{2}\\
f_{\alpha_1}(\pi_5)=:\pi_{7}:&t\to -t\alpha_{1}\\
f_{\alpha_2}(\pi_7)=f_{\alpha_1}(\pi_6)=:\pi_{8}:&t\to -t\lambda
\end{array}
\]
Hence we obtain
$B_{\pi_{\lambda}} = \{ \pi_{\lambda}, \pi_{2}, \pi_{3}, \pi_{4}, \pi_{5}, \pi_{6}, \pi_{7}, \pi_{8}\}$.
\end{example}

\subsubsection{Crystal graphs}
For $\pi$ in $\Pi^{+}$, let $\mathcal{G}(\pi)$ be the oriented  graph with set of vertices  $B_{\pi}$  and an arrow $\pi \overset{i}{\rightarrow} \pi'$ means that $f_{\alpha_i}(\pi) = \pi'$ and $e_{\alpha_i}(\pi') = \pi$. The graph $\mathcal{G}(\pi)$ is also called a \emph{crystal graph},~see~\cite{Kashiwara94}.

For example, consider the semisimple Lie algebra $\mathfrak{g}=\mathfrak{sl}_{3}$. For $\lambda =\omega_{1}+\omega_{2}$, the elements of $B_{\pi_\lambda}$  obtained in~Example~\ref{rootoperatorexample} can be represented in the following crystal graph $\mathcal{G}(B_{\pi_\lambda})$
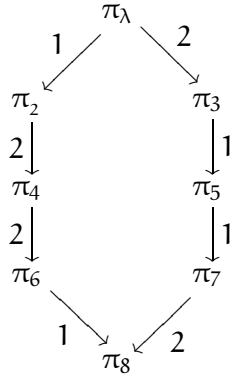
\[
\xymatrix@C=1.5em @R=0.7cm{
 && \pi_{\lambda}
 	\ar@<+0.4ex>[dl] _-{\small 1}
 	\ar@<+0.4ex>[dr]^-{\small 2}
 \\
&\pi_{_2}
 	\ar@<+0.4ex>[d]_-{\small 2}
&& \pi_{3}
\ar@<+0.4ex>[d]^-{\small 1}
\\
&\pi_{4}
\ar@<+0.4ex>[d]_-{\small 2}
&& \pi_{5}
\ar@<+0.4ex>[d]^-{\small 1}
\\
&\pi_{6}
\ar@<+0.4ex>[dr]_-{\small 1}
&&\pi_{7}
\ar@<+0.4ex>[dl]^-{\small 2}
\\
&&\pi_{8}
 } 
\]

\subsubsection{ L-S paths} 
For a dominant weight $\lambda$, the \emph{Lakshmibai-Seshadri paths}, \emph{L-S paths} for short,  of shape $\lambda$ are the paths $\pi$ of the form
\[
\pi= f_{\alpha_1}\circ\ldots\circ f_{\alpha_s}(\pi_\lambda)
\]
where $\alpha_1,\ldots,\alpha_s$ are simple roots of a semisimple Lie algebra $\mathfrak{g}$. That is, these paths are all the elements of $B_{\pi_\lambda}$.

\begin{example} 
Consider $\mathfrak{g}=\mathfrak{sl}_{3}$  and  let $\omega_1$ and $\omega_2$ be its fundamental weights.  Let $\lambda = \omega_1+\omega_2$, the L-S paths of shape $\lambda$ are all the elements of $B_{\pi_\lambda}$ as obtained in~Example~\ref{rootoperatorexample}. The L-S paths of shape $\omega_2$ are the elments of $B_{\pi_{\omega_2}}$, with
\[B_{\pi_{\omega_2}} = \{\;\pi_{\epsilon_1+\epsilon_2}, \pi_{\epsilon_1+ \epsilon_3}, \pi_{\epsilon_2+\epsilon_3} \;\}.\]

For the Lie algebra $\mathfrak{sl}_n$, the L-S paths of shape $\omega_1$ are the paths $\pi_{\epsilon_i}$, for $i=1,\ldots,n+1$.
\end{example}

\subsection{Tableaux}
In this section, we recall the definition of standard tableaux as defined by Littelmann in~\cite{Littelmann96}.
Fix an enumeration $\omega_1, \ldots, \omega_n$ of the fundamental weights of a semisimple Lie algebra~$\mathfrak{g}$.

\subsubsection{Young diagram}
A \emph{Young diagram} is a finite collection of boxes, arranged in left-justified rows such that the length of each row is bigger or equal to the one below it. Note that the rows are enumerated from top to bottom and the columns from the right to the left.
For example, the Young diagram with six boxes in the first row, four boxes in the second boxes and one box in the third row  is the following
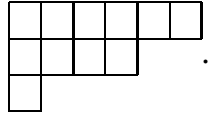
\[\begin{tabular}[c]{|l|lllll}\hline
&\multicolumn{1}{|l|}{}&\multicolumn{1}{|l|}{}&\multicolumn{1}{|l|}{}&\multicolumn{1}{|l|}{}&\multicolumn{1}{|l|}{}\\\cline{1-6}
&\multicolumn{1}{|l|}{}&\multicolumn{1}{|l|}{}&\multicolumn{1}{|l|}{}&\\\cline{1-4}
&\\\cline{1-1}
\end{tabular}.
\]
\subsubsection{L-S monomials}
 Let $\lambda = \lambda_1+\ldots +\lambda_{k}$, where $\lambda_{1},\ldots,\lambda_{k}$ are dominant  weights. If for all $i=1,\ldots , k$, $\pi_{i}$ is an L-S path of shape $\lambda_{i}$, then the monomial $m=\pi_{1}\star\ldots\star\pi_{k}$ is called an \emph{L-S monomial of shape} $\underline{\lambda} = (\lambda_1,\ldots,\lambda_k)$.

\subsubsection{Young tableau}
A \emph{Young tableau of shape} $\lambda= a_{1}\omega_{1}+\ldots+a_{n}\omega_{n}$ is an L-S monomial 
\[
\underset{\underset{1 \leq j \leq n} {1 \leq i \leq n}}{\bigstar}{\pi_{1,\omega_j}\star\ldots\star\pi_{a_i,\omega_j}}
\]
where $\pi_{i,\omega_j}$ is an L-S path of shape $\omega_j$, for $1 \leq i \leq n$. That is, the first $a_{1}$ paths are of shape $\omega_{1}$, the next $a_{2}$ are of shape $\omega_{2}$,\ldots, the final $a_{n}$ paths are of shape $\omega_{n}$.

\subsubsection{Standard tableau} 
Let $m$ be a Young tableau of shape $\lambda = a_{1}\omega_{1}+\ldots+a_{n}\omega_{n}$. The Young tableau $m$ is called $\emph{ standard of shape}$  $\lambda$ if $m\in\mathcal{A}(\pi_{\omega_{1}}\star \ldots\star \pi_{\omega_{n}})$.

\begin{example}
For type $A_n$, consider the ordered alphabet $\mathcal{A}_{n+1}=\{1<\dots<n+1\}$. For a dominant weight $\lambda= a_{1}\omega_{1}+\ldots+a_{n}\omega_{n}$,  set $p_i = a_i+\ldots+a_n$. Consider the Young diagram with $p_1$ boxes in the first row, $p_2$ boxes in the second row, etc. A Young tableau of shape $\lambda$ and type~$A$ is a filling of the boxes of this Young diagram with elements of the alphabet $\mathcal{A}_{n+1}$ such that the entries are striclty increasing in the column from top to bottom. 

A  standard tableau of shape $\lambda= \omega_{1}+ \ldots+ \omega_{k}$ is a Young tableau of shape $\lambda$ and type~A such that the entries are weakly increasing in the rows from left to right. In other words all the standard tableaux of shape $\lambda$ are the elements of $B_{\pi_{\omega_{1}}\star \ldots\star \pi_{\omega_{k}}}$, where  the L-S monomial $\pi_{\omega_{1}}\star \ldots\star \pi_{\omega_{k}}$ corresponds to the Young tableau with only $1$'s in the first row, $2$'s in the second row,\ldots, $k$'s in the k-th row.
Note that the standard tableaux of type~A are also called \emph{semistandard tableaux} in~\cite{Fulton97}.

For type $A_3$,  the following  Young tableau is standard of shape $ \omega_1+2\omega_2+\omega_3$
\[
\young(1112,233,3).
\]
\end{example}

\begin{example}
For type $C_n$, consider the ordered alphabet $\mathcal{C}_n = \{1<\ldots<n<\overline{n}<\ldots<\overline{2}<\overline{1}\}$. For a dominant weight $\lambda= a_{1}\omega_{1}+\ldots+a_{n}\omega_{n}$, set $p_1= a_1+2a_2+\ldots+2a_n$ and for $i\geq 2$ set $p_i=2a_i+\ldots+2a_n$. Consider the Young diagram with $p_1$ boxes in the first row, $p_2$ boxes in the second row, etc. A Young tableau of shape $\lambda$ and type~$C$ is a filling of the boxes of this Young diagram with elements of the alphabet $\mathcal{C}_{n}$ such that the entries are striclty increasing in the column from top to bottom, but $i$ and $\overline{i}$ are never entries in the same column. In addition, for each pair of columns $(C_{a_1+2j-1},C_{a_1+2j})$, $j=1,\ldots, a_2+\ldots+a_{n}$, either these columns are equal or the column $C_{a_1+2j}$ is obtained from $C_{a_1+2j-1}$ by exchanging an even number of times an entry $k$, $1\leq k\leq \overline{1}$, in $C_{a_1+2j-1}$ by $\overline{k}$, see~{\cite[Section~4.2]{GaussentLittelmann12}}.

A standard tableau of shape $\lambda= \omega_{1}+ \ldots+ \omega_{k}$ is a Young tableau of shape $\lambda$ and type~C  such that the entries are weakly increasing in the rows from left to right. Note that the standard tableaux of type~C are also called \emph{ symplectic tableaux} in~\cite{Lecouvey02}.

For type~$C_3$, the following Young tableau is a standard tableau of shape $\omega_1+2\omega_2$
\[\begin{tabular}[c]{|l|l|l|l|l}\hline
$\mathtt{1}$&$\mathtt{1}$&$\mathtt{1}$&$\mathtt{2}$&\multicolumn{1}{|l|}{$\mathtt{\overline{3}}$}\\\cline{1-5}
$\mathtt{2}$&$\mathtt{2}$&$\mathtt{\overline{3}}$&$\mathtt{\overline{3}}$\\\cline{1-4}
$\mathtt{3}$&$\mathtt{3}$&$\mathtt{\overline{2}}$&$\mathtt{\overline{1}}$\\\cline{1-4}
\end{tabular}.\]
\end{example}

\subsection{Plactic algebra for any semisimple Lie algebra}
Let us recall the definition of the plactic algebra for any semisimple Lie algebra and we refer the reader to~\cite{Littelmann96} for more details.

Let $\mathbb{Z}\Pi_{0}$  be the $\mathcal{A}$-submodule $\mathcal{A}\Pi^{+}$ of $\mathbb{Z}\Pi$ generated by the paths in $\Pi^{+}$.
For two paths $\pi_{1}$ and $\pi_{2}$ in $\mathbb{Z}\Pi_{0}$, denote by $\pi_{1}^{+}$ and $\pi_{2}^{+}$ the unique paths in $\Pi^{+}$ such that $\pi_1\in M_{\pi_{1}^{+}}$ and $\pi_{2}\in M_{\pi_{2}^{+}}$. One can define a relation $\sim$ on $\mathbb{Z}\Pi_{0}$ by : $\pi_{1}\sim \pi_{2}$ if, and only if, \textrm{wt}($\pi_{1}^{+}$) = \textrm{wt}($\pi_{2}^{+}$) and 
$\psi(\pi_{1})= \pi_{2}$ under the isomorphism $\psi:M_{\pi_{1}^{+}}\to M_{\pi_{2}^{+}}$. 
The \emph{plactic algebra} for $\mathfrak{g}$ is the quotient 
\[\mathbb{Z}\mathcal{P}:= \mathbb{Z}\Pi_{0}/\sim.\]
For $\pi\in \mathbb{Z}\Pi_{0}$, we denote by $[\pi]\in \mathbb{Z}\mathcal{P}$ its equivalence class. The classes $[m]$ of standard Young tableaux form a basis of the plactic algebra $\mathbb{Z}\mathcal{P}$, see~{\cite[Theorem~7.1]{Littelmann96}}.

\begin{example}[{\cite[Theorem~C]{Littelmann96}}]
For type A, consider the alphabet $\{1,\ldots, n\}$. The plactic congruence coincides with the congruence generated by the following families of relations on the word algebra $\mathbb{Z}\{1,\ldots, n\}$ :
\begin{enumerate}[(a)]
\item  $xzy = zxy$ for  $1\leq x<y\leq z\leq n$.
\item  $yxz = yzx$ for  $1\leq x\leq y < z \leq n$.
\item $12\ldots n$  is the trivial word.  
\end{enumerate}
\end{example}

\section{Finite convergent presentation of plactic monoids}
\label{finiteconvergentpresentation}

\subsection{Rewriting properties of $2$-polygraphs}
In this section, we recall some rewriting properties of the presentations of monoids. These presentations are studied in terms of polygraphs. For more informations, we refer the reader to~\cite{GuiraudMalbos14}.  

A \emph{ $2$-polygraph } (with only one $0$-cell) is a pair  $\Sigma=(\Sigma_{1},\Sigma_{2})$ 
where $\Sigma_{1}$  is the set  of  generating  $1$-cells  and $\Sigma_{2}$ is the set of  generating $2$-cells, that is  rules of the form $u\dfl v$, with $u$ and $v$ are words in the free monoid $\Sigma_{1}^{\ast}$.
A monoid $M$ is presented by a $2$-polygraph $\Sigma$ if $M$ is isomorphic to the quotient of the free monoid $\Sigma_{1}^{\ast}$ by the congruence generated by $\Sigma_{2}$.  
A $2$-polygraph $\Sigma$ is \emph{finite} if $\Sigma_{1}$ and $\Sigma_{2}$ are finite.   A \emph{ rewriting step } of $\Sigma$ is a $2$-cell of the form $w\phi w':wuw'\odfl{}wvw'$,  where $\phi:u\odfl{}v$ is a $2$-cell in $\Sigma_{2}$ and $w$ and $w'$ are words of $\Sigma_{1}^{*}$. A \emph{rewriting sequence} of $\Sigma$ is a finite or infinite sequence of rewriting steps. We say that $u$ rewrites into $v$ if $\Sigma$ has a  nonempty rewriting sequence from $u$ to $v$. A word of $\Sigma_{1}^{\ast}$ is a \emph{ normal form } if $\Sigma$ has no rewriting step with source $u$. A  normal form  of $u$ is a word $v$ of $\Sigma_{1}^{\ast}$ that is a normal form and such that $u$ rewrites into $v$. We say that $\Sigma$ \emph{ terminates } if it has no infinite rewriting sequences.
We say that $\Sigma$ is \emph{ confluent }  if for any words $u$, $u'$ and $u"$ of $\Sigma_{1}^{\ast}$, such that $u$ rewrites into $u'$ and $u"$, there exists a word $v$ in $\Sigma_{1}^{\ast}$ such that $u'$ and $u"$ rewrite into $v$. 
We say that $\Sigma$ is \emph{convergent} if it terminates and it is confluent. Note that a terminating $2$-polygraph is convergent if every word admits a unique normal form.

\subsection{Column presentation}
Let $\mathfrak{g}$ be a semisimple Lie algebra and let $\omega_{1},\ldots,\omega_{n}$ be its fundamental weights.  Let $\mathbb{B}_{i}$ be the set of L-S paths of shape $\omega_{i}$ and $\mathbb{B}=\cup_{i=1}^{n}\mathbb{B}_{i}$. 
Denote by $\mathbb{B}^{\ast}$ the free monoid over $\mathbb{B}$. A word on~$\mathbb{B}^{\ast}$ is a concatenation of L-S paths.

For every L-S paths $c_{1}$ and $c_{2}$ in $\mathbb{B}$ such that $c_{1}\star c_{2}$ is not a standard tableau, we define a $2$-cell
\[
 c_{1}\star c_{2} \odfl{\gamma_{c_1,c_2}} T 
\]
where $T$ is the unique standard tableau such that $[T] = [c_{1}\star c_{2}]$.
 
Let denote by $\Path(n)$ the $2$-polygraph whose set of $1$-cells is $\mathbb{B}$ and the set of $2$-cells is 
\[
\Path_2(n)
\: = \:
\big\{ \; c_{1}\star c_{2} \odfl{\gamma_{c_1,c_2}} T \; \big| \; c_1,c_2\in\mathbb{B}\;\text{and}\; \; c_{1}\star c_{2} \text{ is not a standard tableau } \big\}.
\]
This presentation is called the \emph{column presentation}. It is a presentation of the plactic monoid for any semisimple Lie algebra~{\cite[Theorem~B]{Littelmann96}}. Let us prove that it is terminating and confluent.

\subsubsection{Order on the tableaux}
Consider the partial order $\leqslant$ of dominant weights defined by 
\[\lambda_1 \leqslant \lambda_2 \qquad \text{ if and only if }\qquad \lambda_1-\lambda_2\in \mathbb{N}\Phi^{+}\]
where $\lambda_1$ and $\lambda_2$ are dominant weights. That is, $\lambda_1-\lambda_2$ is a nonnegative integral sum of positive roots. Using this order, one  can find for each dominant weight, a finite number of dominant weights that are smaller than it,  then the partial order $\leqslant$ is a well-founded order.

Let us define an order $\preceq$ on the tableaux of $\mathbb{B}^{\ast}$ as follows. For two tableaux $m$ and $m'$ of shape $\lambda$ and $\lambda'$ respectively, we have
\[
 m \preceq m' \qquad \text{if and only if }\qquad \lambda \leqslant \lambda'.
\]

\begin{remark}
\label{mainremark}
In the proof of Theorem~B in~\cite{Littelmann96}, Littelmann showed  that for every tableau $m$ of shape $\lambda$ that it is not standard, the shape of the standard tableau $m'$ such that $[m]=[m']$ is strictly smaller than $\lambda$. 
\end{remark}

\begin{lemma}
\label{terminationLemma}
The $2$-polygraph $\Path(n)$ is terminating.
\end{lemma}

\begin{proof}
Let us show that the $2$-polygraph $\Path(n)$ is compatible with the order~$\preceq$.
We have to prove that  if  $h \odfl{ } h'$, then $h'\prec h$, where $h$ is  not a sandard tableau and $h'$ is a standard tableau of shape~$\lambda'$. There is two cases depending on whether or not $h$ is a tableau. 

\noindent Suppose that $h$ is a tableau of shape $\lambda$ that it is not standard, then by Remark~\ref{mainremark} we have $\lambda'<\lambda$. Thus we obtain that $h'\prec h$.

\noindent Suppose that $h$ is an L-S monomial of shape $\underline{\lambda}$ that it is not a tableau. By decomposing each L-S path of shape $\omega_i$, for $i=1,\ldots,n$, in $h$ into concatenation of L-S paths of shape $\omega_1$,  we transform the L-S monomial $h$ into a tableau of shape $k\omega_1$, with $k\in\mathbb{N}$. Again by Remark~\ref{mainremark}, we have $\lambda'< k\omega_1$. Thus we obtain that $h'\prec h$.

Hence, rewriting an L-S monomial that is not a standard tableau always decreases it with repsect to the order $\preceq$.  Since every application of a $2$-cell of $\Path(n)$ yields a $\prec$-preceding tableau, it follows that any sequence of rewriting using $\Path(n)$ must terminate.
\end{proof}

\begin{lemma}
\label{ConfluenceLemma}
The $2$-polygraph $\Path(n)$ is confluent.
\end{lemma}

\begin{proof}
Let $m\in \mathbb{B}^{*}$ and  $T$, $T'$ be two normal forms such that  $m\odfl{} T$ and $m\odfl{} T'$. It is sufficient to prove that $T=T'$. Suppose $T=c_1\star\ldots\star c_k$, where the L-S monomial $c_1\star\ldots\star c_k$ is a standard  Young tableau such that $[m] = [c_1\star\ldots\star c_k]$.
Similarly, $T' = c'_{1}\star\ldots\star c'_{l}$  where the L-S monomial $c'_{1}\star\ldots\star c'_{l}$ is a standard  Young tableau such that
$[m] = [c'_1\star\ldots\star c'_l]$.
Since $[m] = [c_1\star\ldots\star c_k] =  [c'_1\star\ldots\star c'_l]$ and the standard tableaux form a cross-section of the plactic monoid for a semisimple Lie algebra $\mathfrak{g}$, we have $k=l$ and $c_{i}=c'_{i}$, for all $i=1,\ldots, k.$ Thus $T=T'$. Since the $2$-polygraph $\Path(n)$ is terminating, and rewriting any non-standard tableau must terminate with a unique normal form, $\Path(n)$ is confluent.
\end{proof}

Every semisimple Lie algebra $\mathfrak{g}$ admits a finite number of fundamental weights, then there is a finite number of L-S paths of shape~$\omega_i$, for $i=1,\ldots,n$. Thus the $2$-polygraph $\Path(n)$ is finite. Hence, by Lemmas~\ref{terminationLemma} and~\ref{ConfluenceLemma}, we obtain the following theorem.

\begin{theorem}
\label{maintheorem}
For any semisimple Lie algebra $\mathfrak{g}$,  the $2$-polygraph $\Path(n)$ is a finite convergent presentation of the plactic monoid for $\mathfrak{g}$.
\end{theorem}

\subsection{Finiteness properties of plactic monoids}
A monoid is of \emph{finite derivation type} ($\mathrm{FDT}_3$) if it admits a finite presentation whose relations among the relations are finitely generated, see~\cite{Squier94}. The property $\mathrm{FDT}_3$ is a natural extension of the properties of being finitely generated $(\mathrm{FDT}_1)$ and finitely presented $(\mathrm{FDT}_2)$. Using the notion of polygraphic resolution,  one can define the higher-dimensional finite derivation type properties $\mathrm{FDT}_{\infty}$, see~\cite{GuiraudMalbos12advances}. They generalise in any dimension the finite derivation type $\mathrm{FDT}_{3}$.  A monoid is said to be $\mathrm{FDT}_{\infty}$ if it admits a finite polygraphic resolution. By Corollary~4.5.4 in~\cite{GuiraudMalbos12advances}, a monoid with a finite convergent presentation is $\mathrm{FDT}_{\infty}$. Then by Theorem~\ref{maintheorem}, we have

\begin{proposition}
For any semisimple Lie algebra $\mathfrak{g}$ , plactic monoids for $\mathfrak{g}$  satisfy the homotopical finiteness condition $\mathrm{FDT}_{\infty}$.
\end{proposition}

In the homological way, a monoid $M$ is of \emph{homological type} $\mathrm{FP}_{\infty}$ when there exists a resolution of $M$ by projective, finitely generated $\mathbb{Z}M$-modules. By Corollary~5.4.4 in~\cite{GuiraudMalbos12advances} the property $\mathrm{FDT}_{\infty}$ implies the property $\mathrm{FP}_{\infty}$. Hence we have

\begin{proposition}
For any semisimple Lie algebra $\mathfrak{g}$ , plactic monoids for $\mathfrak{g}$ satisfy the homological finiteness property type $\mathrm{FP}_{\infty}$.
\end{proposition}


\begin{small}
\renewcommand{\refname}{\Large\textsc{References}}
\bibliographystyle{alpha}
\bibliographystyle{plain}
\bibliography{biblioLittelmannPaths}
\end{small}
\vfill

\noindent \textsc{Nohra Hage}\\
\begin{small}
Universit\'e de Lyon,\\
Institut Camille Jordan, CNRS UMR 5208\\
Universit\'e Jean Monnet\\
23, boulevard du docteur Paul Michelon,\\
42000 Saint-Etienne cedex, France.\\
\textsf{ nohra.hage@univ-st-etienne.fr}
\end{small}

\end{document}